\documentclass[a4paper,10pt]{article}

\usepackage[utf8x]{inputenc}
\usepackage[english]{babel}

\usepackage{amsmath}
\usepackage{graphicx}
\usepackage{hyperref}

\usepackage{amsfonts}
\usepackage{amsthm}
\usepackage{amssymb}
\usepackage{cite} 

\newcommand{\R}{\mathbb{R}} 
\newcommand{\N}{\mathbb{N}} 
\newcommand{\C}{\mathbb{C}}

\newcommand{\graph}{\operatorname{graph}}
\newcommand{\supp}{\operatorname{spt}}

\newtheorem{theorem}{Theorem}[section]
\newtheorem{lemma}[theorem]{Lemma}

\theoremstyle{definition}
\newtheorem{definition}[theorem]{Definition}
\newtheorem{remark}[theorem]{Remark}

\numberwithin{equation}{section}

\title{Lower regularity assumption for an Euler-Lagrange equation on the contact line of the phase dependent Helfrich energy.}

\author{Sascha Eichmann\\
Mathematisch-Naturwissenschaftliche Fakultät,\\
Eberhard Karls Universität Tübingen,\\
Auf der Morgenstelle 10,\\
D-72076 Tübingen, Germany\\
E-mail: \href{mailto:sascha.eichmann@math.uni-tuebingen.de}{sascha.eichmann@math.uni-tuebingen.de}\\
Phone: +49/7071/2976886}

\begin{document}
 \maketitle

\begin{abstract}
We examine the phase dependent Helfrich energy and show an Euler-Lagrange equation on the phase seperation line.
This result has already been observed by e.g. Jülicher-Lipowski and later Elliot-Stinner. Here we are able to lower the regularity assumption for this result down to $C^{1,1}$ for the seperation line.\\
In the proof we employ a carefully choosen test function utilising the signed distance function.
\end{abstract}

\section{Introduction}
\label{sec:1}
The Canham-Helfrich energy (or short Helfrich energy) is defined for a two dimensional smooth, orientable surface $S\subset \R^3$ with mean curvature $H$ by
\begin{equation}
 \label{eq:1_1}
 W_{H_0}(S) = \int_S (H-H_0)^2\, dA.
\end{equation}
Here $H_0\in\R$ is called spontaneous curvature. This energy is used in e.g. modelling the shape of lipid bilayers, see e.g. \cite{Helfrich} or red blood cells, see \cite{Canham}. In more layman's terms lipid bilayers form the boundary of depots in biological cells. These kind of depots are called vesicles. The thickness of these boundaries is usually small compared to the whole vesicle, hence modelling the shape of it by a two-dimensional surface is feasible (see figure \ref{fig_1}).
\begin{figure}
\begin{center}
 \includegraphics{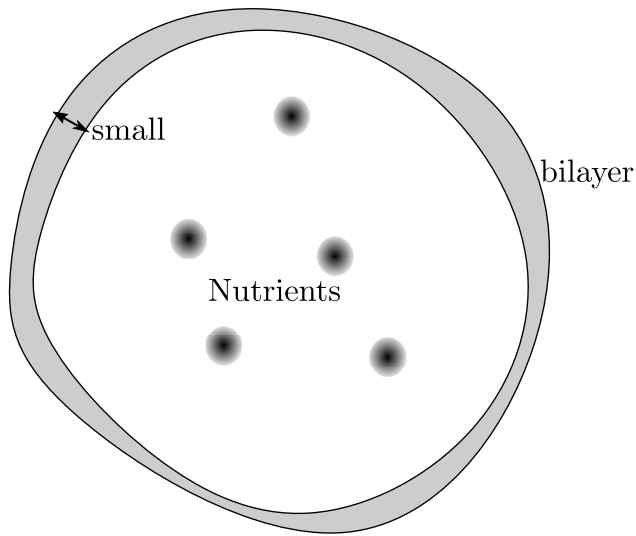}
 \end{center}
 \caption{Sketch of a vesicle with homogeneous lipid bilayer.}
 \label{fig_1}
\end{figure}

The parameter $H_0$ represents an asymmetry of the lipid bilayer. I.e. it has been observed, that it has a prefered curvature, which is dependent on the material the bilayer consists of (see e.g. \cite{LipowskiReview} and the references therein).

Since the bilayer itself does not consist of a homogeneous material, it may happen that the spontaneous curvature differs throughout the bilayer. Then it has been observed that two phases of the bilayer form, each having their own prefered spontaneous curvature. These different phases usually do not mingle, but rather form a sharp contact line (see e.g. \cite{BaumgartHessWebb} for some experimental results or \cite{HuWeiklLip} for some numerical simulations). 
The Helfrich energy has been adapted to model this kind of behaviour in \cite{JuelichLip1} and \cite{JuelichLip2}.

Let us describe this model now (cf. figure \ref{fig_2}).
\begin{figure}
\begin{center}
 \includegraphics{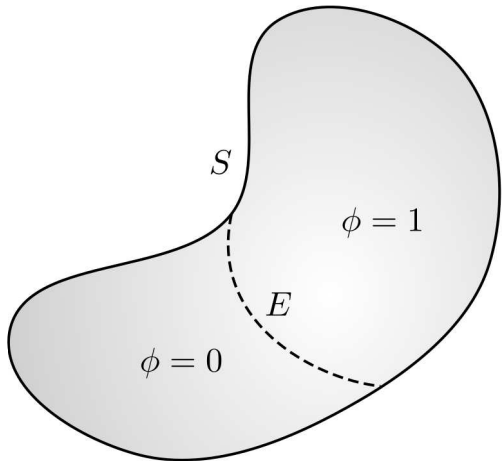}
 \end{center}
 \caption{Sketch of a surface with domain/phase seperation.}
 \label{fig_2}
\end{figure}
Let $S\subset \R^3$ be an oriented compact two dimensional surface with scalar mean curvature $H$. Furthermore we describe the two different lipid bilayers by a function $\phi:S\rightarrow\{0,1\}$, i.e. $x\in S$ belongs to the bilayer of type $i\in\{0,1\}$, iff $\phi(x)=i$. We call this type phase or domain. The spontaneous curvature is now a function depending on the phase, i.e. $H_0:\{0,1\}\rightarrow \R$. 
Since the phases have been seperated, it is natural to assume, that their contact is minimal. Hence the length of the contact line should be minimised as well (see e.g. \cite{JuelichLip1} and the references therein). For that let $J\subset S$ be the jump set of $\phi$ and let it be a one-dimensional curve (in the Hausdorff sense, see e.g. \cite[§2.1]{EvansGariepy} and \cite[§2]{Simon_Buch} for additional informations). Then the domain dependend Helfrich energy (or phase dependent Helfrich energy) is defined by
\begin{equation}
 \label{eq:1_2}
 W_{H_0,\sigma}(S,\phi) = \int_S(H(x)-H_0(\phi(x)))^2\, dA + \sigma\mathcal{H}^1(J).
\end{equation}
We call the term $\sigma\mathcal{H}^1(J)$ line tension. Here $\sigma\in\R$ is a parameter describing the contribution of the length of $J$.

The aim of this article is to calculate an Euler-Lagrange equation on $J$. The theorem for $S$ being in $C^1$ is as follows:
\begin{theorem}
 \label{1_1}
 Let $S\subset R^3$ be an oriented surface with mean curvature $H$ and phase seperation $\phi:S\rightarrow\{0,1\}$. We assume $S$ to be a $C^1$ surface.
 Furthermore $S|_{\{\phi=i\}}$ is supposed to be in $C^4$, such that the derivatives up to third order can be extended to the contact line $J$ (i.e. the jump set of $\phi$) for $i=0,1$. 
 Now we additionally assume $(S,\phi)$ to be critical for $W_{H_0,\sigma}$, i.e. for any smooth vectorfield $V\in \C^\infty_c(\R^3,\R^3)$ and their associated flow $\Phi:\R^3\times \R\rightarrow\R$ we have
 \begin{equation*}
  0=\frac{d}{dt}W_{H_0,\sigma}(\Phi(S,t),\phi(\Phi(\cdot,t))).
 \end{equation*}
Finally we assume $J$ to be of class $C^{1,1}$. 
We call $H^i$ the extension of the scalar mean curvature of $S|_{\{\phi=i\}}$ to $J$. This extension satisfies 
\begin{equation}
\label{eq:1_3}
 H^0(x) - H_0(0) = H^1(x)-H_0(1)
\end{equation}
for all $x\in J$.
\end{theorem}

Unfortunately there are some experimental results in \cite{BaumgartHessWebb}, which indicate, that $S$ is not necessarily $C^1$, but only $C^0$. In this case \eqref{eq:1_3} becomes:
\begin{equation}
\label{eq:1_3_1}
 H^i(x) - H_0(i) = 0\mbox{ for all }x\in J\mbox{ and }i=0,1.
\end{equation}
A precise statements for $S$ being a graph is given in Theorem \ref{1_2_1}.

The equalities \eqref{eq:1_3} and \eqref{eq:1_3_1} have first been observed in \cite{JuelichLip1} and \cite{JuelichLip2} for $S$ and $J$ being axially symmetric.
Unfortunately as seen in the experiments conducted in \cite{BaumgartHessWebb} and the numeric in \cite{HuWeiklLip} this symmetry cannot be expected in general.
Equations \eqref{eq:1_3} and \eqref{eq:1_3_1} have been extended to the more general case of surfaces in \cite[Problem 2.16]{ElliottStinner_2010} (see also \cite[Eq. (4.12)]{Wutz_2017} for a more background information).
There $J$ is usually assumed to be smooth. Here we will lower this requirement to $C^{1,1}$.\\
The proof involves finding a suitable test function for the first variation. In this argument the signed distance function (and its regularity) of $J$ will be paramount. There we will follow an argument given by Foote in \cite{Foote}.\\

Existence of minimisers of the phase dependend Helfrich energy were first shown in \cite{ChoksiMorandottiVeneroni} for axially symmetric surfaces. Later in \cite{BrazdaLussardiStefanelli} this was extended to the more general class of curvature varifolds. Unfortunately regularity for this varifold minimiser and the corresponding contact line is still mainly open.
Our Theorem \ref{1_1} shows that equality \eqref{eq:1_3} (rsp. \eqref{eq:1_3_1}) has to be incorporated into such a proof, since it is a necessary condition for regularity. How to do this and if (and how) one can show \eqref{eq:1_3} for such a minimizer, are open questions for future research.

A phase field approach of the phase dependend Helfrich energy for axially symmetric surfaces and the corresponding analysis can be found in \cite{Helmers1} and \cite{Helmers2}. A parametric finite element approach to the corresponding flow of axially symmetric surfaces has been done recently in \cite{GarckeNuernberg}.

Further numerical studies for more general surfaces for the corresponding flow have been conducted in e.g. \cite{BarretGarckeNuernberg} with a phase field apprach and \cite{BarrettGarckeNuernberg_2} in a sharp interface setting. In \cite[§3]{BarrettGarckeNuernberg_2} a weaker version of the flow was formulated, which does not need as much regularity. Furthermore an analogue formulation for flows of \eqref{eq:1_3} has been given in \cite[Eq. (1.4a)]{BarrettGarckeNuernberg_2}.

Without the phase dependency the available literatur is quite vast and developing quite rapidly. A calculation of the Euler-Lagrange equation without phase dependencies has already been done in \cite{OuYang_Helfrich}. There are also several existence results, e.g. in the axially symmetric case \cite{ChoksiVeroni}, \cite{DeckDoemGrunau}, for graphs \cite{DeckGruRoe} and for compact immersions \cite{MondinoScharrer}, \cite{EichmannHelfrichClosed}. Further some other modifications to the Helfrich energy have been analysed, e.g. adding an elastic energy for the boundary in \cite{PalmerPampano}. Also some stability results in \cite{Lengeler} and \cite{BernardWheeler} are available.

If we additionally consider $H_0=0$ our energy becomes the famous Willmore energy. We refer to the survey articles \cite{KuwertSchaetzleSurvey}, \cite{GrunauWillmoreSurvey}, \cite{HellerPedit}, the seminal paper of Willmore \cite{Willmore} and the proof of the Willmore conjecture \cite{NevesMarques} in this case.

\subsection{Strategy of the proof}
\label{sec:1_1}
Since equation \eqref{eq:1_3} is of local nature, we can assume without loss of generality $S$ to be a graph of a function $u:\Omega\rightarrow\R$ with $\Omega\subset \R^2$ a smooth domain. Let us further refine the phase seperation $\phi$ in this case: Let $A_i\subset \Omega$ be the maximal open set on which $\phi(\graph u|_{A_i})=i$. For the proof let $E$ be the contact line in the parameter region $\Omega$, i.e. $E:=\Omega\setminus(A_0\cup A_1)$ (or the projection of $E$ onto $\Omega$). Then $E$ is a $C^{1,1}$ curve by our assumptions on $J$ in Theorem \ref{1_1}.
In this case we call $A_0,A_1$ a phase seperation of $\Omega$ with contact line $E$. 
The phase dependend Helfrich energy for graphs is
\begin{align}
 \label{eq:1_4}
 \begin{split}
 &W_{H_0,\sigma}(u, (A_0,A_1))\\
 :=& \int_\Omega (H_u(x)- H_0(\phi(u(x),x))^2\sqrt{1+|\nabla u(x)|^2}\, dx + \sigma\mathcal{H}^1(\graph u|_E).
\end{split}
 \end{align}

We reformulate Theorem \ref{1_1} for graphs:
\begin{theorem}
\label{1_2}
 Let $\Omega\subset \R^2$ be an open bounded set. Furthermore let $A_0,A_1\subset \Omega$ be a phase seperation of $\Omega$ with contact line $E$. Now let $u\in C^1(\Omega)$ and
 $u|_{A_i}\in C^{4}(A_i)\cap C^3(\overline{A_i})$ be a critical point of $u\mapsto W_{H_0,\sigma}(u, (A_0,A_1))$, i.e. for all $\varphi\in C^\infty_0(\Omega)$ we have
 \begin{equation*}
  0=\frac{d}{dt}W_{H_0,\sigma}(u + t\varphi, (A_0,A_1))|_{t=0}.
 \end{equation*}
Furthermore we assume $E$ to be of class $C^{1,1}$.\\
Then the derivatives of $u$ up to the order of $3$ can be extended continuously to $E$. We call this extension $u_{A_i}$ and the corresponding extension of the mean curvature of $\graph u$ is $H_{u_{A_i}}$. Then for every $x\in E$ we have
\begin{equation*}
 H_{u_{A_0}}(x) - H_0(0) =  H_{u_{A_1}}(x) - H_0(1).
\end{equation*}
\end{theorem}
The formulation of \eqref{eq:1_3_1} as a Theorem for graphs is as follows:
\begin{theorem}
\label{1_2_1}
 Let $\Omega\subset \R^2$ be open and bounded. Furthermore let $A_0,A_1\subset \Omega$ be a phase seperation of $\Omega$ with contact line $E$. Now let $u\in C^0(\Omega)$ and
 $u|_{A_i}\in C^{4}(A_i)\cap C^3(\overline{A_i})$ be a critical point of $u\mapsto W_{H_0,\sigma}(u, (A_0,A_1))$, i.e. for all 
 \begin{align*}
  \varphi\in &\{\psi:\Omega\rightarrow\R|\ \psi|_{A_i}\in C^\infty(A_i)\mbox{ s.t. }D^k\psi\\ &\mbox{ is uniformely continuous for }k=1,2,3\mbox{ and }\operatorname{supp}\psi\subset \Omega\}
 \end{align*}
 we have
 \begin{equation*}
  0=\frac{d}{dt}W_{H_0,\sigma}(u + t\varphi, (A_0,A_1))|_{t=0}.
 \end{equation*}
Furthermore we assume $E$ to be of class $C^{1,1}$.\\
Then the derivatives of $u$ up to the order of $3$ can be extended continuously to $E$. We call this extension $u_{A_i}$ and the corresponding extension of the mean curvature of $\graph u$ is $H_{u_{A_i}}$. Then for every $x\in E$ and $i=0,1$ we have
\begin{equation*}
 H_{u_{A_i}}(x) - H_0(i) = 0.
\end{equation*}
\end{theorem}

The proof of Theorem \ref{1_2} requires two ingredients: First we calculate an Euler-Lagrange equation in section \ref{sec:2}, which will give us a condition on the boundary with test functions in $C^1$. The other step is showing the existence of a suitable test function itself, which will boil down to showing the signed distance function for $E$ is $C^1$ close to and on $E$. We will demonstrate this in section \ref{sec:3}. 
Section \ref{sec:4} is dedicated to bringing all arguments together to finish the proof. There we also explain the changes needed to obtain Theorem \ref{1_2_1}

\begin{remark}
 \label{1_3}
 Since $E$ is of Class $C^{1,1}$ the corresponding phase seperation $A_0,A_1$ of $\Omega$ satisfies
 \begin{equation*}
  \mathcal{L}^2(\Omega\setminus (A_0\cup A_1))=0\mbox{ and }A_0\cap A_1=\emptyset.
 \end{equation*}
This will allow us to seperate the area integral in \eqref{eq:1_4} in two distinct parts. 
\end{remark}


\section{A necessary condition for criticality}
\label{sec:2}
Under the assumptions of Theorem \ref{1_2}, we will show the following Euler-Lagrange equation on $E$ (see e.g. \cite{ElliottStinner_2010}) for a function $\varphi\in C^1_0(\Omega)$, which is zero on $E$.
\begin{align}
\label{eq:2_1}
\begin{split}
 &0=\int_E\left(\langle \nabla \varphi,n_E\rangle - \sum_{i,j=1}^2\frac{(n_E)_i\partial_j \varphi\partial_i u \partial_j u }{1+|\nabla u|^2} \right)\\&\cdot\left(H_{u_{A_0}} - H_0(0) - (H_{u_{A_1}} - H_0(1))\right)\, d\mathcal{H}^1.
 \end{split}
\end{align}
Here $n_E$ is a normal of $E$. In case the $A_i$ are only Caccioppoli it is the corresponding measure theoretic normal (see \cite[p. 169]{EvansGariepy}).
By Remark \ref{1_3} we have $\mathcal{L}^2(\Omega\setminus(A_0\cup A_1))=0$ and $A_0\cap A_1=\emptyset$.
Hence
\begin{align*}
W_{H_0,\sigma}(u,(A_0,A_1)) =& \sigma\mathcal{H}^1(\graph u|_E)\\
& + \sum_{i=1}^2\int_{A_i} \left(H_u - H_0(i)\right)^2\sqrt{1+|\nabla u|^2}\,dx
\end{align*}
and therefore we only need to calculate the derivative of the bulk term on $A_i$.
First we assume $\varphi$ to be smooth. Afterwards we get our result by approximation.
Let us start with the mean curvature of the variation (see e.g. \cite[Eq. (8)]{DeckGruRoe} for a formula for the mean curvature of graphs):
\begin{align*}
 H_{u+t\varphi} =& \sum_{i=1}^2\partial_i\left(\frac{\partial_i(u+t\varphi)}{\sqrt{1+\sum_j(\partial_ju + t\partial_j\varphi)^2}}\right)\\
 =& \sum_{i=1}^2\frac{\partial_i^2(u+t\varphi)}{\sqrt{1+|\nabla(u+t\varphi)|^2}}\\
 & - \sum_{i=1}^2\frac{\partial_i(u+t\varphi)\sum_{j=1}^2(\partial_ju+ t\partial_j\varphi)(\partial_i\partial_j u + t\partial_i\partial_j\varphi)}{\left(1+\sum_{j=1}^2(\partial_j u + t\partial_j\varphi)^2\right)^\frac{3}{2}}
\end{align*}
Hence
\begin{align*}
 \frac{d}{dt}H_{u+t\varphi}|_{t=0}=& \sum_i\frac{\partial_i^2 \varphi}{\sqrt{1+|\nabla u|^2}} - \sum_{i,j}\frac{\partial_i^2u\partial_j\varphi\partial_ju}{(1+|\nabla u|^2)^\frac{3}{2}}\\
 &-\sum_{i,j}\frac{\partial_i\varphi\partial_j u\partial_{i}\partial_j u + \partial_iu\partial_j\varphi\partial_i\partial_j u + \partial_i u\partial_j u \partial_i\partial_j \varphi}{(1+|\nabla u|^2)^\frac{3}{2}}\\
 & + \sum_{i,j,k}\frac{3\partial_i u \partial_j u \partial_i\partial_j u\partial_k\varphi\partial_k u}{(1+|\nabla u|^2)^\frac{5}{2}}.
\end{align*}
Since we assume boundedness for $u$ and the derivatives, the dominated convergence theorem yields
\begin{align*}
 &\frac{d}{dt}\int_{A_i} |H_{u+t\varphi} - H_0(i)|^2\sqrt{1+|\nabla (u +t\varphi)|^2}\, dx|_{t=0}\\
 =& \int_{A_i}2(H_u-H_0(i))\frac{d}{dt}(H_u+t\varphi)|_{t=0}\sqrt{1+|\nabla u|^2}\\
 & + |H_u - H_0(i)|^2\frac{\langle \nabla u,\nabla \varphi\rangle}{\sqrt{1+|\nabla u|^2}}\, dx
\end{align*}
Now using partial integration twice, we obtain as the bulk term the usual Helfrich equation, which is zero by only utilising test functions from $C^\infty_0(A_i)$, and the rest of the boundary terms. Almost all of these boundary terms are zero, because $\varphi=0$ on $E$. Only the ones remain, where originally the second derivatives of $\varphi$ were present:
\begin{align*}
 &\frac{d}{dt}\int_{A_i} |H_{u+t\varphi} - H_0(i)|^2\sqrt{1+|\nabla (u +t\varphi)|^2}\, dx|_{t=0}\\
 =& 2\int_E \left(\langle\nabla \varphi, n_E\rangle - \sum_{i,j}\frac{\partial_j\varphi\partial_i u\partial_j u (n_E)_i}{1+|\nabla u|^2}\right)(H_{u_{A_i}}-H_0(i))\, d\mathcal{H}^1\lfloor E.
\end{align*}
Since we have this term twice with opposite signs (we use the same normal $n_E$ for both $A_0$ and $A_1$ in the partial integration), we obtain \eqref{eq:2_1}.

\section{The signed distance function}
\label{sec:3}
In this section let $c:[a,b]\rightarrow \R^2$ be an injective, regular curve with $c\in C^{1,1}$ and being parametrised by arclength.
We will show, that there exists a neighbourhood of $c$, such that the signed distance function is in $C^1$.
We will follow the presentation of \cite{Foote}. First we start by showing that the projection is well defined:
\begin{lemma}
 \label{3_1}
 For all $x\in c((a,b))$ exists a $\delta>0$ such that the Projection 
 \begin{equation*}
  P:B_{\frac{\delta}{2}}(x)\rightarrow c((a,b))\cap B_\delta(x)
 \end{equation*}
is well defined, i.e. for all $y\in B_{\frac{\delta}{2}}(x)$ the nearest point on the curve $c$ is unique.
\end{lemma}
\begin{proof}
 We proceed by contradiction and assume we find a sequence $y_n\in B_{\frac{1}{2n}}(x)$ and $t_n^1<t_n^2\in (a,b)$, such that
 \begin{equation*}
  \operatorname{dist}(y_n, c((a,b))\cap B_{\frac{1}{n}}(x)) = |c(t_n^1)-y_n| = |c(t^2_n)-y_n|
 \end{equation*}
For $i=1,2$ we have
\begin{align}
\begin{split}
\label{eq:3_0_1}
 0=& \partial_t|c(t)-y_n|^2|_{t=t_n^i}\\
 =& 2\langle\dot{c}(t_n^i),c(t_n^i)\rangle - 2\langle \dot{c}(t_n^i),y_n\rangle
\end {split}
 \end{align}
and therefore $\dot{c}(t_n^i)$ is orthogonal to $c(t_n^i)-y_n$. 
Let $\nu:(a,b)\rightarrow \mathbb{S}^1$ be a choosen lipschitz continuous unit normal of $c$ with Lipschitz constant $L>0$.
Then we find $\lambda_n^i$, satisfying
\begin{equation*}
 c(t_n^i)-y_n = \lambda_n^i \nu(t_n^i) 
\end{equation*}
and therefore
\begin{equation*}
 \lambda_n^i = \pm \operatorname{dist}(y_n, c((a,b))\cap B_{\frac{1}{n}}(x)).
\end{equation*}
If we choose $n$ big enough, the curve $c$ seperates $B_{\frac{1}{n}}(x)$ into two connected regions (see Figure \ref{fig_3}), since $c\in C^{1,1}$.
\begin{figure}
\begin{center}
 \includegraphics{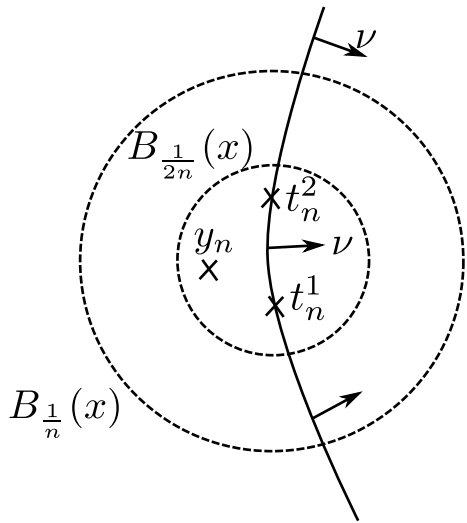}
 \end{center}
 \caption{Sketch of nonuniqueness situation with normals.}
 \label{fig_3}
\end{figure}

Hence we can assume without loss of generality
\begin{equation*}
 \lambda_n^i =  \operatorname{dist}(y_n, c((a,b))\cap B_{\frac{1}{n}}(x)).
\end{equation*}
Hence $0<\lambda_n^i\leq \frac{1}{n}$, independent of $i$ (i.e. $\lambda_n=\lambda_n^1=\lambda_n^2$) and furthermore
\begin{equation*}
 c(t_n^i)-y_n=\lambda_n \nu(t_n^i).
\end{equation*}
Subtracting these two equations yields
\begin{equation*}
 c(t_n^1) -c(t_n^2) = \lambda_n \nu(t_n^1) - \lambda_n\nu(t_n^2). 
\end{equation*}
By the mean value theorem we find $\xi_{n,1},\xi_{n,2}\in (t_n^1,t_n^2)$ with
\begin{equation*}
 \left|\left(\begin{array}{c}\dot{c}^1(\xi_{n,1})\\ \dot{c}^2(\xi_{n,2})\end{array}\right)\right||t_n^1-t_n^2| = \lambda_n|\nu(t_n^1)-\nu(t_n^2)|\leq  L\lambda_n|t_n^1 - t_n^2|.
\end{equation*}
Since $c$ is parametrised by arclength and $|t_n^1-t_n^2|\rightarrow 0$, we have for $n\rightarrow\infty$
\begin{equation*}
 \left|\left(\begin{array}{c}\dot{c}^1(\xi_{n,1})\\ \dot{c}^2(\xi_{n,2})\end{array}\right)\right|\rightarrow 1.
\end{equation*}
Hence there is a constant $\varepsilon>0$ and $n_0=n_0(\varepsilon)\in\N$ such that for every $n\geq n_0$ we have
\begin{equation*}
 \left|\left(\begin{array}{c}\dot{c}^1(\xi_{n,1})\\ \dot{c}^2(\xi_{n,2})\end{array}\right)\right|\geq 1-\varepsilon.
\end{equation*}
All in all we get
\begin{equation*}
 1-\varepsilon \leq \frac{1}{n}L
\end{equation*}
for every $n\geq n_0$, which is a contradiction.
\end{proof}

\begin{remark}
 \label{3_1_1}
 Sets on which the projection is locally unique are called sets with positive reach. This notion goes back to \cite{FedererCurvature} and has been characterised very well recently in \cite{RatajZaj} (see also the references therein). In our case the arguments are a lot simpler, hence we provided them for the sake of completeness in Lemma \ref{3_1}.\\
 The characterisation in \cite{RatajZaj} is essentially a $C^{1,1}$ regularity assumption for the set in question. In this sense our method is optimal, though we do not know, whether the regularity assumption of $E$ in our theorems is optimal as well.
\end{remark}

The next lemma shows, that the projection $P$ is continuous. The proof given here is by Federer \cite[4.8(4)]{FedererCurvature}
\begin{lemma}[see \cite{FedererCurvature}]
 \label{3_2}
 Let $x\in c((a,b))$ and $\delta>0$ as in Lemma \ref{3_1}. Then the projection $P$ on the curve $c$
 \begin{equation*}
  P:B_{\frac{\delta}{2}}(x)\rightarrow c((a,b))\cap B_\delta(x)
 \end{equation*}
is continuous.
\end{lemma}
\begin{proof}
 Let us assume the opposite. Thereby we find a sequence $y_i\in B_\frac{\delta}{2}(x)$ with $y_i\rightarrow y\in B_\frac{\delta}{2}(x)$
  and an $\varepsilon>0$ such that
  \begin{equation*}
   |P(y)-P(y_i)|\geq \varepsilon
  \end{equation*}
for every $i\in\N$.
Since $P(y_i)\in B_\delta(x)$ the sequence $P(y_i)$ is bounded. After choosing a subsequence and relabeling, we get an $a\in \overline{B_\delta(x)}\cap c([a,b])$ with $P(y_i)\rightarrow a$.
Furthermore the distance function is Lipschitz and therefore we obtain
\begin{equation*}
 \frac{\delta}{2}>\operatorname{dist}(y,c((a,b))) = \lim_{i\rightarrow\infty}\operatorname{dist}(y_i,c((a,b))) = \lim_{i \rightarrow\infty}|P(y_i)-y_i| = |a-y|.
\end{equation*}
Hence $a\in B_\delta(x)$.
By Lemma \ref{3_1} the projection is unique and we get
\begin{equation*}
 a=P(y).
\end{equation*}
This yields by the convergence of $P(y_i)\rightarrow a$
\begin{equation*}
 |P(y) - P(y_i)| = |a-P(y_i)|\rightarrow 0
\end{equation*}
and therefore the desired contradiction.
\end{proof}

Let us now define a signed distance function:
\begin{definition}
 \label{3_3}
 Let $x\in c((a,b))$ and $\delta>0$ as in the preceeding lemmas. Then for every $y\in B_\frac{\delta}{2}(x)$ the projection is unique and there exists a unique $\lambda_y\in \R$ (cf. \eqref{eq:3_0_1}), such that
 \begin{equation*}
  P(y)-y = \lambda_y \nu(P(y)).
 \end{equation*}
Here $\nu$ is an a priori choosen Lipschitz continuous unit normal of $c$.
Then we define a signed distance function $d_\pm:B_\frac{\delta}{2}(x)\rightarrow\R$ by
\begin{equation*}
 d_\pm(y)= \left\{\begin{array}{cc}\operatorname{dist}(y,c((a,b))),&\mbox{ if } \lambda_y\geq0 \\
                  -\operatorname{dist}(y,c((a,b))),&\mbox{ if }\lambda_y < 0.
                 \end{array}\right.
\end{equation*}
\end{definition}

\begin{remark}
 \label{3_4}
The sign of $d_\pm$ switches, if the other unit normal is choosen. 
\end{remark}

Now we can show the central result of this section. For the readers benefit we provide the proof already given in \cite[Thm. 2]{Foote} here as well.
\begin{theorem}[see \cite{Foote} Thm. 2]
 \label{3_5}
 Let $\nu$ be a lipschitz continuous unit normal of $c$. Then for every $x\in c((a,b))$ there exists a $\delta>0$, 
 such that the signed distance function choosen with respect to $\nu$ satisfies $d_\pm\in C^1(B_\frac{\delta}{2}(x))$. Furthermore the derivative satisfies
 \begin{equation*}
  \nabla d_\pm(y) = \left\{\begin{array}{cc}\frac{y-P(y)}{|y-P(y)|},&\mbox{ if } \lambda_y>0\\
                            -\frac{y-P(y)}{|y-P(y)|},&\mbox{ if } \lambda_y<0\\
                            \nu(y),&\mbox{ if }y\in c((a,b)).
                           \end{array}\right.
 \end{equation*}
Here $\lambda_y$ is defined as in Definition \ref{3_3}.
\end{theorem}
\begin{proof}
 We show the theorem first for $y\in B_\frac{\delta}{2}(x)\setminus c((a,b))$. Let $d:=(d_\pm)^2$ and 
 now we proceed by contradiction and assume there exists a $v\in \R^2$ such that
 \begin{equation}
 \label{eq:3_2}
  \liminf_{t\searrow 0}\frac{d(y+tv) - d(y)}{t} < 2\langle y - P(y),v\rangle
 \end{equation}
or
 \begin{equation}
 \label{eq:3_3}
  \limsup_{t\searrow 0}\frac{d(y+tv) - d(y)}{t} > 2\langle y - P(y),v\rangle.
 \end{equation}
 Let us work through case \eqref{eq:3_2} first:
 Hence we find an $\varepsilon>0$, a $t>0$ arbitrarily small with
 \begin{equation*}
  d(y+tv) < d(y) + 2\langle y-P(y),tv\rangle - t\varepsilon
 \end{equation*}
Hence 
\begin{align*}
 & |y-P(y+tv)|^2 = |y+tv - P(y+tv) - tv|^2\\
 =& d(y+tv) - 2\langle y+tv-P(y+tv),tv\rangle + t^2|v|^2\\
 <&d(y) + 2\langle P(y+tv) -P(y),tv\rangle - t^2|v|^2 - t\varepsilon
\end{align*}
Since $P$ is continuous we have $2\langle P(y+tv) -P(y),tv\rangle = o(|t|)$ and hence this term can be absorbed for small $t$ and we get
\begin{equation*}
 |y-P(y+tv)|^2 < d(y) - t^2|v|^2 - t\frac{\varepsilon}{2} < d(y)=|y-P(y)|^2
\end{equation*}
and therefore $P(y+tv)$ has a strictly smaller distance to $y$ than $P(y)$, which is a contradiction.\\
Now we assume \eqref{eq:3_3}: As in the preceeding case, there is an $\varepsilon>0$ and $t>0$ arbitrarily small, such that
\begin{equation*}
 d(y+tv)>d(y) + 2\langle y-P(y),tv\rangle + t\varepsilon.
\end{equation*}
Hence
\begin{align*}
 |y+tv-P(y)|^2 =& |y-P(y)|^2 + t^2|v|^2 + 2\langle y-P(y),tv\rangle\\
 <& d(y+tv) - t\varepsilon + t^2|v|^2.
\end{align*}
By choosing $t$ small enough, we obtain
\begin{equation*}
 |y+tv-P(y)|^2 < d(y+tv).
\end{equation*}
Therefore $P(y)$ is a strictly better projection than $P(y+tv)$, i.e. a contradiction.
Hence for all $y\in B_\frac{\delta}{2}(x)\setminus c((a,b))$ we have
\begin{equation*}
 \nabla d(y) = 2(y-P(y)).
\end{equation*}
Since $d_\pm(y)= \operatorname{sign}(\lambda_y)\sqrt{d(y)}$ we obtain by chain rule for these $y$:
\begin{equation*}
 \nabla d_\pm(y) = \operatorname{sign}(\lambda_y)\frac{\nabla d(y)}{2\sqrt{d(y)}} = \lambda_y\frac{y-P(y)}{|y-P(y)|}.
\end{equation*}
Let us now examine $y\in c((a,b))$: 
Let $y_i\in B_\frac{\delta}{2}(x)\setminus c((a,b))$ with $y_i\rightarrow y$. Then for each $y_i$ exists a neighbourhood, such that $\lambda_{(\cdot)}$ does not change sign. Then by $|\lambda_{y_i}|=|y_i-P(y_i)|$ we get
\begin{align*}
 \nabla d_\pm(y_i) =& \operatorname{sign}(\lambda_{y_i})\frac{y_i-P(y_i)}{|y_i-P(y_i)|}= \operatorname{sign}(\lambda_{y_i})\frac{\lambda_{y_i}}{|y_i-P(y_i)|}\nu(P(y_i))\\= &\operatorname{sign}(\lambda_{y_i})^2\nu(P(y_i))\rightarrow \nu(y),
\end{align*}
since $P$ is continuous and $P(y)=y$. Hence $d_\pm$ is continuously differentiable.
\end{proof}

\section{Finishing the proofs}
\label{sec:4}
In the first part of this section we show Theorem \ref{1_2}. After that we highlight the changes we need to make to obtain a proof of Theorem \ref{1_2_1}.
\begin{proof}

Let $\varphi \in C^\infty_0(\Omega)$ be arbitrary. Since it has compact support and $E$ is relatively compact in $\Omega$ (because $E$ is a $C^{1,1}$ curve) the intersection $\supp\varphi\cap E$ is compact as well.
Using a partition of zero and Theorem \ref{3_5} we find $\sigma_i\in C^\infty_0(B_{\delta_i}(x_i))$ ($i=1,\ldots,N$) with $x_i\in \supp\varphi\cap E$, $\delta_{i}>0$, such that
the signed distance function $d_\pm$ with respect to $E$ and the choosen normal $n_E$ satisfies $d_\pm\in C^1(B_{\delta_{i}}(x_i))$.
Furthermore $\sum_i\sigma_i=1$.
Then we can define 
\begin{equation}
\label{eq:4_0}
 \Phi:=\varphi\sum_i d_\pm \sigma_i\in C^1_0(\Omega),
\end{equation}
where we implicitly continue $d_\pm$ by zero to the whole of $\Omega$. Furthermore $\Phi$ can be plugged into \eqref{eq:2_1}, because $d_\pm$ is zero on $E$. We also obtain for any $y\in E$:
\begin{equation*}
 \nabla\Phi(y) = \varphi(y)n_E(y)\sum_i\sigma_i(y)=\varphi(y)n_E(y). 
\end{equation*}
Hence \eqref{eq:2_1} yields
\begin{equation*}
 0=\int_E \varphi\left(1 - \sum_{ij}\frac{(n_E)_i(n_E)_j\partial_i u\partial_j u}{1+|\nabla u|^2}\right)(H_{u_{A_0}} - H_0(0) - (H_{u_{A_1}} - H_0(1))\, d\mathcal{H}^1 
\end{equation*}
Let us now analyse $1 - \sum_{ij}\frac{(n_E)_i(n_E)_j\partial_i u\partial_j u}{1+|\nabla u|^2}$:
 \begin{align*}
  &\sum_{ij}(n_E)_i(n_E)_j\partial_i u\partial_j u \\
  =& ((n_E)_1)^2(\partial_1u)^2 + 2(n_E)_1(n_E)_2\partial_1 u\partial_2u + ((n_E)_2)^2(\partial_2 u)^2\\
  =& ((n_E)_1\partial_1 u + (n_E)_2\partial_2 u)^2 = (\langle n_E,\nabla u\rangle)^2\leq |n_E|^2|\nabla u|^2=|\nabla u|^2
 \end{align*}
and hence
\begin{equation*}
 1 - \sum_{ij}\frac{(n_E)_i(n_E)_j\partial_i u\partial u_j}{1+|\nabla u|^2}>0.
\end{equation*}
Since $\varphi$ itself is arbitrary, the fundamental lemma of variational calculus yields
\begin{equation*}
 H_{u_{A_0}} - H_0(0) - (H_{u_{A_1}} - H_0(1))=0\quad \mathcal{H}^1\lfloor E\mbox{-a.e.}
\end{equation*}
Since the mean curvatures on $E$ are continuous continuations, we obtain the desired conclusion.
\end{proof}

\begin{remark}
\label{4_1}
 Adding $\lambda_1\operatorname{area}(S)+ \lambda_2\operatorname{vol}(S)$ to the phase dependend Helfrich energy \eqref{eq:1_2} to account for additional conditions like prescribing the area and/or enclosed volume of $S$, does not change the result of Theorem \ref{1_1}. 
 Since these terms do not produce terms with derivatives of the test function in the first variation, the proof of Theorem \ref{1_1} remains valid. 
\end{remark}

Now we explain the changes to the proof to obtain Theorem \ref{1_2_1}
\begin{proof}
 Since $u\in C^0(\Omega)$ and $C^4(A_i)\cap C^3(\overline{A_i})$, we can choose from a greater set of testfunctions $\psi$. We employ a function, which is zero on $A_1$ and $E$, i.e. we take $\psi\in C^0_0(\Omega)$ with $\psi\in C^1(A_i)$, such that the derivative can be extended to $E$.
 Then the same calculations to obtain \eqref{eq:2_1} yield
 \begin{equation}
 \label{eq:4_1}
  \int_E\left(\langle \nabla \psi,n_E\rangle - \sum_{i,j=1}^2\frac{(n_E)_i\partial_j \varphi\partial_i u_0 \partial_j u_0 }{1+|\nabla u_0|^2} \right)\left(H_{u_{A_0}} - H_0(0)\right)\, d\mathcal{H}^1=0.
 \end{equation}
Here $\nabla u_0$ refers to the continuation of $\nabla u|_{A_0}$ to $E$.
Instead of using the signed distance function to construct a suitable $\psi$, we employ the following function instead
\begin{equation*}
 d_0(y):=\left\{\begin{array}{cc} \operatorname{dist}(y,E),& y\in A_0\cup E\\0,&\mbox{else.}\end{array}\right. 
\end{equation*}
This $d_0$ is continuous and the same techniques employed in Lemma \ref{3_5} show $d_0$ to be $C^1$ close to $E$ with continued derivative
\begin{equation*}
\nabla d_0(y)=\pm n_E(y) \mbox{ for }y\in E.
\end{equation*}
The sign depends on whether $n_E$ points inward or outward of $A_0$.
As in \eqref{eq:4_0} we can now construct a suitable test function $\psi$ with an arbitrary $\varphi\in C^\infty_0(\Omega)$. This we can plug in \eqref{eq:4_1} and obtain
\begin{equation*}
 0=\int_E \varphi\left(\pm1 \mp \sum_{ij}\frac{(n_E)_i(n_E)_j\partial_i u_0\partial_j u_0}{1+|\nabla u_0|^2}\right)(H_{u_{A_0}} - H_0(0))\, d\mathcal{H}^1.
\end{equation*}
As in the proof of Theorem \ref{1_2} the prefactor is never zero and then the fundamental lemma of variational calculus yields \eqref{eq:1_3_1}.
The same arguments can be employed to show the result for the continuation of $u|_{A_1}$ on $E$.
\end{proof}

\phantomsection 
\addcontentsline{toc}{section}{References}
\bibliography{bibliography.bib}
\bibliographystyle{plain}


\end{document}